\documentclass[11pt]{article}

\usepackage[T1]{fontenc}
\usepackage{lmodern}

\usepackage{amsmath, amssymb, amsthm}
\usepackage{mathtools}
\usepackage{cite}
\usepackage[all]{xy}
\usepackage{hyperref}

\numberwithin{equation}{section}

\newtheorem{theorem}{Theorem}[section]
\newtheorem{proposition}[theorem]{Proposition}
\newtheorem{lemma}[theorem]{Lemma}
\newtheorem{corollary}[theorem]{Corollary}

\newtheorem{definition}[theorem]{Definition}
\newtheorem{example}[theorem]{Example}
\newtheorem{remark}[theorem]{Remark}
\newtheorem{question}[theorem]{Question}

\newcommand{\id}{\mathrm{id}}

\newcommand\Z{\mathbb{Z}}
\newcommand\Q{\mathbb{Q}}
\newcommand\R{\mathbb{R}}
\newcommand\T{\mathbf{T}}
\newcommand\pT{{\T^*}}

\newcommand\Se{\mathbf{S}}

\newcommand\pS{{\Se^*}}

\newcommand\A{\mathbf{A}}
\newcommand\TA{\mathbf{TA}}
\newcommand\Hom{\mathrm{Hom}}

\title{Weak-split extensions of topological Abelian groups}

\author{
Mar\'{\i}a V. Ferrer$^{1}$,
Salvador Hern\'andez-Mu\~noz$^{1}$, \\
Luis Javier Hern\'andez-Paricio$^{2}$\thanks{Corresponding author: \texttt{luis-javier.hernandez@unirioja.es}}
}

\date{}

\begin{document}
\maketitle

\vspace{-0.8cm}

\begin{center}
{\small
$^{1}$ Departamento de Matem\'aticas and IMAC, Universitat Jaume I, Castell\'on, Spain\\
$^{2}$ Departamento de Matem\'aticas y Computaci\'on, Universidad de La Rioja, Logro\~no, Spain
}
\end{center}

%\begin{center}
%{\small
%$^{1}$ Departamento de Matem\'aticas and IMAC, Universitat Jaume I, Castell\'on, Spain
%
%$^{2}$ Departamento de Matem\'aticas y Computaci\'on, Universidad de La Rioja, Logro\~no, Spain
%}
%\end{center}

%\vspace{0.1cm}

\noindent\textbf{Published version:}
Electronic Research Archive, 34 (2026), 1720--1741.\\
\url{https://doi.org/10.3934/era.2026078}

\vspace{0.5cm}

\begin{abstract}
In the category of topological Abelian groups, we consider the usual notion of extension $E=(B\to X \to A)$ of $B$ by $A$ and the notion of weak-split extension (when $X\to A$  has a continuous split $A \to X$). Given a weak-split extension $E$, the topological Abelian group $X$ is homeomorphic to $B \times A$, but in general, $X$ need not be algebraic isomorphic to $B \times A$. In this paper, for two topological Abelian groups $A, B$, we study the Abelian group $E_{\TA}^{ws}(A,B) $ of weak-split extensions of $B$ by $A$ modulo extension isomorphisms. We prove that  $E_{\TA}^{ws}(A,B) $ can be described as  the Abelian group  of  all possible continuous sums given in  the product  topological space  $B \times A$ (modulo topological isomorphism) having $B$ as a topological subgroup and $A$ as a topological quotient.  We also give an alternative description of $E_{\TA}^{ws}(A,B) $ as a quotient $\mathcal{Z}_c(A,B)/\mathcal{B}_c(A,B)$, where $\mathcal{Z}_c(A,B)$ are cocycles  represented  by certain continuous maps of the form $A\times A \to B$,  and similarly for the coboundaries $\mathcal{B}_c(A,B)$. For two topological Abelian groups $A,B$, we compare the Abelian group of $ws$-extensions $E_{\TA}^{ws}(A,B)$ with the Abelian group of standard extensions $E_{\A}(A,B)$ where now  $A, B$  also denote the subjacent Abelian groups. We relate these different types of extensions using an exact sequence with six terms.
Although the Bohr topology of discrete Abelian groups has been investigated by many workers, there still remain
many parts that are not well  understood. Here, as an application of the methods developed in the paper, new examples of nontrivial $ws$-extensions for discrete Abelian groups equipped with the Bohr topology are provided and some related open questions are also proposed.
\end{abstract}



\section{Introduction}

For given topological Abelian groups $A$ and $B$, an extension of $B$ by $A$ in the category of topological Abelian groups $\TA$ is a sequence of topological Abelian groups
\[
\xymatrix{ B \ar[r]^j & X \ar[r]^p & A,}
\]
where $p\colon X\to A$ is a surjective open homomorphism and $j \colon B\to X$ is the kernel of $p$ in $\TA$. 

For given Abelian groups $C$ and $D$, an extension of $D$ by $C$ is a sequence in the category of Abelian groups $\A$:
\[
\xymatrix{ D \ar[r]^i & Y \ar[r]^q & C,}
\]
where $q$ is a surjective homomorphism and $i\colon D\to Y$ is the kernel of $q$ in $\A$.

In this paper, we use the following notation: $\T$ is the category of topological spaces, $\pT$ is the category of pointed topological spaces, $\Se$ is the category of sets, and $\pS$ is the category of pointed sets. 
Given a category $\mathbf{C}$ and two objects $A, B$ in $\mathbf{C}$, we denote by $\Hom_{\mathbf{C}}(A,B)$ the set of morphisms from $A$ to $B$, and by $\mathrm{Aut}_{\mathbf{C}}(A)$ the set of automorphisms of the object $A$.

We consider extensions (short exact sequences) of the following types:
\begin{itemize}
  \item[(i)] Algebraic extensions, $\xymatrix{ D \ar[r]^{i} & Y \ar[r]^{q} & C,}$ in $\A$.
  \item[(ii)] Extensions of topological Abelian groups, $\xymatrix{ B \ar[r]^{j} & X \ar[r]^{p} & A,}$ in $\TA$.
  \item[(iii)] Weakly-split extensions of topological Abelian groups, $$\xymatrix{ B \ar[r]^{j} & X \ar[r]^{p} & A,}$$ which are extensions in $\TA$ such that $p$ admits a continuous section; that is, there exists a pointed continuous map $\mathbf{s} \colon A \to X$ (in $\pT$) such that $p\mathbf{s} = \id_A$ (in certain instances, we will use bold letters to emphasize that some continuous maps need not be homomorphisms). 
\end{itemize}

The extensions of type (i) have been extensively studied in the literature: the extensions of $D$ by $C$, modulo extension equivalence, have the structure of an Abelian group with the Baer sum \cite{Baer29,Baer34}. Standard references on extensions of Abelian groups and homological algebra include \cite{Eilenberg42,Hilton-Stammbach,Fuchs2015}.

Extensions of type (ii) have also been studied in a topological context; see for instance \cite{Hugo,BELLO20161607,BELLO2017379}.

A systematic analysis of weak-split extensions of type (iii) (called $ws$-extensions in this paper) was initiated in the doctoral thesis of Bello~\cite{Hugo}, where ws-extensions were introduced and studied from the perspective of local splitting and pseudo-homomorphisms. Subsequent developments by Bello et al.~\cite{BELLO20161607} provide criteria for the existence of continuous or local cross sections and characterize when an extension splits.

For $ws$-extensions of type (iii) in which $B$ is a closed subgroup of $X$, it is said in the literature that $B$ is a css-subgroup of $X$ (see also Definition~\ref{ccs}). From this perspective, several results on css-subgroups can be interpreted as structural results on $ws$-extensions. Nevertheless, there exist $ws$-extensions where $B$ is not closed in $X$; see Example~\ref{ej1}. For further results on css-subgroups and Bohr topologies, see \cite{COMFORT2001215}.

Given two $ws$-extensions 
\[
E=\left(\xymatrix{ B \ar[r]^{j} & X \ar[r]^{p} & A}\right)
\quad \text{and} \quad
E'=\left(\xymatrix{ B \ar[r]^{j'} & X' \ar[r]^{p'} & A}\right),
\]
we say that $E \sim E'$ if there exists a continuous homomorphism $\theta \colon X \to X'$ such that $\theta j = j'$ and $p' \theta = p$. This generates an equivalence relation on $ws$-extensions of $B$ by $A$, and the quotient is denoted by $E_{\TA}^{ws}(A,B)$. For analogous constructions in cases (i) and (ii), see the references cited above.

The main goal of this paper is to study $E_{\TA}^{ws}(A,B)$, endowed with an Abelian group structure via the Baer sum. Our main results are presented in Theorem~\ref{sumproduct}, Theorem \ref{cocycles}, and Theorem  \ref{sixterms}. In Theorems~\ref{sumproduct} and \ref{cocycles}, we extend classical algebraic results (cf.~\cite[Chapter 9]{Fuchs2015}) to the topological setting, while Theorem~\ref{sixterms} compares the continuous and discrete situations.

A fundamental property of a weak-split extension 
\[
\xymatrix{ B \ar[r]^{j} & X \ar[r]^{p} & A}
\]
in $\TA$ is that $X$ is always homeomorphic to the product space $B \times A$.

We first consider the quotient set $\Sigma_{A,B}/\sim$ of all topological Abelian group structures $+'_{B \times A}$ on $B \times A$, up to topological isomorphism, such that $(B,+)$ is a topological subgroup and $(A,+)$ is a quotient group. Each such structure determines a weak-split extension
\[
E_{+'_{B \times A}} = \left(\xymatrix{ (B, +) \ar[r]^-{i_1} & (B \times A, +'_{B \times A}) \ar[r]^-{p_2} & (A,+)}\right),
\]
where $i_1(b)=(b,0)$ and $p_2(b,a)=a$.

We prove that the canonical map
\[
\Sigma_{A,B}/\sim \to E_{\TA}^{ws}(A,B), \quad [+'_{B \times A}] \mapsto [E_{+'_{B \times A}}],
\]
is an isomorphism of Abelian groups (Theorem~\ref{sumproduct}). We also show (Theorem~\ref{cocycles}) that $E_{\TA}^{ws}(A,B)$ is isomorphic to the quotient $\mathcal{Z}_c(A,B)/\mathcal{B}_c(A,B)$ of suitable cocycles and coboundaries.

Finally, for given topological Abelian groups $A,B$, we obtain in Theorem~\ref{sixterms} a six-term exact sequence relating $E_{\TA}^{ws}(A,B)$ and $E_{\A}(A,B)$.

In the last section, we present examples and computations, showing in particular that $E_{\TA}^{ws}(A,B)$ can be nontrivial.

In the algebraic setting, for the case of (non-Abelian) groups, there are analogous results that represent extensions of the form $N \to Z \to G$, where $N$ is a normal Abelian subgroup, as elements of the cohomological group $H^2(G,N)$; see \cite{Eilenberg42,Brown82}. In the non-Abelian case, the study of extensions of topological groups with continuous cocycles requires more sophisticated techniques involving actions of $G$ on $N$ and non-Abelian cohomologies. However, the advantage of our setting is that it admits a useful simplification that connects directly with Abelian homological algebra.

\section{Weak-split extensions in $\TA$}

In this work, we consider the canonical forgetful functors
$U \colon \TA \to \pT$ and $W\colon \TA \to \A$.
Given a topological Abelian group $A$, we sometimes consider only its algebraic structure $WA$, and at other times its underlying pointed topological space $UA$. To avoid a lengthy notation, when the context is clear, both $UA$ and $WA$ will be denoted simply by $A$.

\medskip

\begin{remark}
Denote by $\mathbf{HTA}$ the full subcategory of Hausdorff topological Abelian groups. The kernel of a morphism $f\colon A \to B$ in $\mathbf{HTA}$ is the same as that in the category $\TA$. However, the cokernel of $f$ in $\mathbf{HTA}$ is given by $B/\mathrm{cl}(f(A))$, where $\mathrm{cl}(f(A))$ is the topological closure of $f(A)$. Note that the inclusion functor $\mathbf{HTA} \to \TA$ does not preserve cokernels.
\end{remark}

Recall that a map $f \colon X \to Y$ between topological spaces is open if $f(O)$ is open in $Y$ for every open set $O \subseteq X$. The map $f$ is said to be relatively open if its corestriction $X \to f(X)$, $x \mapsto f(x)$, is open.

\begin{definition}\label{ws-ext}
A sequence
\[
E=(\xymatrix{ B \ar[r]^j & X \ar[r]^p & A })
\]
in $\TA$ is said to be an extension of $B$ by $A$ if $p$ is a quotient map and $(B, j)$ is the kernel of $p$. Equivalently, if $j$ and $p$ are relatively open maps.
\end{definition}

\medskip

\begin{lemma}\label{sectionretraction}
The following statements are equivalent for an extension
\[
E=(\xymatrix{ B \ar[r]^j & X \ar[r]^p & A })
\]
in $\TA$.
\begin{itemize}
\item[(i)] There is a continuous pointed retraction $\mathbf{r}\colon X \to B$ ($\mathbf{r} j=\id_B$ in $\pT$, $\mathbf{r}(0)=0$), such that $\mathbf{r}(x+j(b))=\mathbf{r}(x)+b$ for every $x\in X$ and $b\in B$.
\item[(ii)] There is a continuous pointed section $\mathbf{s} \colon A \to X$ in $\pT$ ($p\mathbf{s}=\id_A$, $\mathbf{s}(0)=0$).
\item[(iii)] There exists a pointed homeomorphism $\boldsymbol{\phi} \colon B \times A \to X$ such that the following diagram is commutative in $\pT$:
\[
\xymatrix{
B \ar[r]^{i_1} \ar[d]_{\id_B} & B \times A \ar[r]^{p_2} \ar[d]_{\boldsymbol{\phi}} & A \ar[d]^{\id_A} \\
B \ar[r]^j & X \ar[r]^p & A
}
\]
where $i_1(b)=(b,0)$ and $p_2(b,a)=a$.
\end{itemize}
\end{lemma}

\begin{proof}
For a complete proof we refer the reader to \cite[Lemma 5.1.5]{Hugo}.
\end{proof}

\begin{remark}\label{homeomorphism}
Suppose that
\[
E=(\xymatrix{ B \ar[r]^j & X \ar[r]^p & A })
\]
is an extension in $\TA$.
\begin{itemize}
\item[(i)] If $\mathbf{s} \colon A \to X$ is a continuous pointed section of $p$, the associated continuous pointed retraction $\mathbf{r}\colon X \to B$ is given by $\mathbf{r}=j^{-1}(\id_X-\mathbf{s}p)$ and satisfies $\mathbf{r}(x+j(b))=\mathbf{r}(x)+b$ for every $x\in X$ and every $b \in B$.
\item[(ii)] If $\mathbf{r}\colon X \to B$ is a continuous pointed retraction such that $\mathbf{r}(x+j(b))=\mathbf{r}(x)+b$ for every $x\in X$ and every $b \in B$, then the associated section $\mathbf{s}$ is given by $\mathbf{s}(p(x))= x-j\mathbf{r}(x)$.
\item[(iii)] The homeomorphism $\boldsymbol{\phi} \colon B \times A \to X$ induced by $\mathbf{s}$ is given by $\boldsymbol{\phi}(b,a)=j(b)+\mathbf{s}(a)$ $(b,a) \in B \times A$, and its inverse is $\boldsymbol{\varphi}(x)=(\mathbf{r}(x),p(x))$,  $x \in X$.
\end{itemize}
\end{remark}

\begin{definition}\label{wses}
An extension
\[
E=(\xymatrix{ B \ar[r]^j & X \ar[r]^p & A })
\]
in $\TA$ is said to be a \emph{weak-split extension} (or a \emph{$ws$-extension}) if there exists a continuous pointed section $\mathbf{s} \colon A \to X$ such that $p\mathbf{s}=\id_A$ and $\mathbf{s}(0)=0$. We also say that $\mathbf{s}$ is a \emph{continuous split} and   that sequence $E$ \emph{$ws$-splits}.
\end{definition}

\begin{remark}\label{nows}  Let $B, A$ be topological Abelian groups. In general, an extension $B \to X \to A$ need not have a continuous pointed section. For instance, consider the extension $E=(\mathbb{Z} \to \mathbb{R} \to \mathbb{T})$, where $\mathbb{Z}$ is the discrete topological Abelian group  of integer numbers, $\mathbb{R}$ is the  topological Abelian group  of real numbers with the standard topology and $\mathbb{T}$ is the 1-sphere (also called 1-torus)  with the standard topology.
 Since $\mathbb{R}$ is not homeomorphic to $\mathbb{Z} \times \mathbb{T}$, it follows that $E$     does not admit a pointed continuous section.
 
 However, in Section 5, we present examples of nontrivial extensions that admit a  continuous pointed section; for instance, taking Bohr topologies,  the extension $({\mathbb{Z}}^\sharp \to {\mathbb{R}}^\sharp \to {\mathbb{T}}^\sharp )$ admits a continuous section. See Proposition \ref{R1} in  Section 5.
 
 In some cases,  under appropriate  topological conditions on
 $B, A$, every extension is a weak-split extension.
For example, if $B$ is $T_2$-compact  and $A$ is a zero-dimensional $k_\omega$-space, then every extension of the form $B \to X \to A$ admits a continuous pointed section; see  \cite[Corollary  5.1.16]{Hugo} or
\cite[Theorem 2.8]{BELLO20161607}.
\end{remark}

 We can consider the category $\mathcal{E}_{\TA} (A,B)$,
whose objects are extensions (see De\-fi\-ni\-tion \ref{ws-ext}) $$E=(\xymatrix{  B \ar[r]^\alpha& X \ar[r]^\beta & A }),$$ and a morphism $\theta \colon E \to E'$ is given by a  continuous homomorphism $\theta \colon X \to X'$ such that the following diagram is commutative:
$$\xymatrix{  \ar[d]_{\id_B} B \ar[r]^\alpha &  \ar[d]^{\theta}  X \ar[r]^\beta & A   \ar[d]^{\id_A}\\
    B \ar[r]^{\alpha'}& X' \ar[r]^{\beta' } & A. }$$

In this work, we are interested in the study of the full subcategory  $\mathcal{E}_{\TA}^{ws} (A,B)$ of  $\mathcal{E}_{\TA} (A,B)$
determined by all the      weak-split extensions; see Definition \ref{wses}.
As a consequence of the following Lemma \ref{isom},  every morphism in $\mathcal{E}_{\TA}^{ws} (A,B)$ is a isomorphism. This implies that
$\mathcal{E}_{\TA}^{ws} (A,B)$ is a large groupoid; that is, a category such  that every morphism is an isomorphism.

In a similar way for the category of Abelian groups, if $A, B \in \A$,  we can  consider the large groupoid $\mathcal{E}_{\A} (A,B)$ of extensions of $B$ by $A$. For two topological Abelian groups $A, B \in \TA$, we note the existence of a natural forgetful functor $\mathcal{E}_{\TA} (A,B) \to \mathcal{E}_{\A} (A,B)$.
\medskip

  \begin{lemma}\label{isom}  Consider the following commutative diagram in $\TA$:
   $$\xymatrix{  \ar[d]_{\id_B} B \ar[r]^\alpha &  \ar[d]^{\theta}  X \ar[r]^\beta & A   \ar[d]^{\id_A}\\
    B \ar[r]^{\alpha'}& X' \ar[r]^{\beta' } & A }$$

   Suppose that $E=(\xymatrix{B \ar[r]^{\alpha}& X \ar[r]^\beta & A})$ and  $E'=(\xymatrix{B \ar[r]^{\alpha'}& X' \ar[r]^{\beta'} & A})$ are extensions. Then, we  have the following:
   \begin{itemize}
  \item[(i)] If $E$ is a weak-split extension and  ${\bf{s}} \colon A \to X$ is a  continuous pointed section, then
  $E'$ is also a weak-split extension, and ${\bf{s'}}=\theta {\bf{s}}  \colon A \to X'$ is a weak-split of $E'$.
    \item[(ii)]   If $E, E'$ are weak-split extensions and  $\theta$ is a continuous homomorphism, then $\theta$ is a topological  isomorphism (a homeomorphism).
\end{itemize}
   \end{lemma}

   \begin{proof} (i): Note that $\beta'{\bf{s'}}=\beta'\theta {\bf{s}}=\beta{\bf{s}}=\id_A$.

(ii) Consider the commutative diagram:      $$\xymatrix{ \ar[d]_{\id_B}  B \ar[r]^i & \ar[d]_{\boldsymbol{\phi}}   B \times A \ar[r]^{p_2} & A \ar[d]^{\id_A} \\
    \ar[d]_{\id_B} B \ar[r]^{\alpha} & \ar[d]_{\theta}  X \ar[r]^{\beta} & A \ar[d]^{\id_A} \\
 B \ar[r]^{\alpha'} & X' \ar[r]^{\beta'} & A },$$

where from Remark  \ref{homeomorphism},
$\boldsymbol{\phi} \colon B \times A \to X$, $\boldsymbol{\phi}((b,a))=\alpha (b) + {\bf{s}}(a))$, and
$\boldsymbol{\phi}' \colon B \times A \to X'$, $\boldsymbol{\phi}'((b,a))=\alpha' (b) + {\bf{s}'}(a))$ are homeomorphisms. Since $\theta$ is a continuous homomorphism and the commutativity of the diagram,  we also have the following:

 $\theta({\boldsymbol{\phi}}(b,a))=\theta(\alpha (b)+{\bf{s}}(a))=\theta \alpha (b)+\theta{\bf{s}}(a)=\alpha'(b)+{\bf{s}}'(a)={\boldsymbol{\phi}'}(b,a).  $

 This implies that $\theta={\boldsymbol{\phi}'}{\boldsymbol{\phi}}^{-1}$ is a topological  isomorphism (a homeomorphism).
  \end{proof}

Given two extensions \(E, E' \in \mathcal{E}_{\TA}(A,B)\), we say that
\(E \sim E'\) if there exists a morphism
\(\theta \colon E \to E'\) in \(\mathcal{E}_{\TA}(A,B)\).
This relation generates an equivalence relation defined as follows:
We say that \(E \approx E'\) if there exists a finite sequence
\[
E_0, E_1, \ldots, E_n \in \mathcal{E}_{\TA}(A,B),
\]
such that \(E = E_0\), \(E' = E_n\), and for each
\(i \in \{1,\ldots,n\}\) either \(E_{i-1} \sim E_i\) or
\(E_i \sim E_{i-1}\).

We have analogous relations in the full subcategory
\(\mathcal{E}_{\TA}^{ws}(A,B)\).
However, in the case of \(ws\)-extensions, Lemma~\ref{isom} implies that
\(E \sim E'\) if and only if \(E \approx E'\).
We also remark that the categories
\(\mathcal{E}_{\TA}(A,B)\) and
\(\mathcal{E}_{\TA}^{ws}(A,B)\) have finite skeletons.
Therefore, we may consider the following quotient~sets.

\begin{definition}
Given \(A, B \in \TA\), the quotient set
\[
E_{\TA}(A,B)
  = \mathcal{E}_{\TA}(A,B)/\!\approx
\]
is called the set of equivalence classes of extensions of \(B\) by \(A\)
in the category of topological Abelian groups modulo continuous homomorphisms.
 Analogously,
\[
E_{\TA}^{ws}(A,B)
  = \mathcal{E}_{\TA}^{ws}(A,B)/\!\approx
\]
is the set of equivalence classes of \(ws\)-extensions of \(B\) by \(A\)
modulo continuous homomorphisms.
\end{definition}

Similarly, given two extensions  $E, E'$ in  $\mathcal{E}_{\A} (A,B)$, we denote by $E_{\A}(A, B)$ the set of all equivalence classes of extensions of Abelian groups of the form
$B \to X \to A$ modulo extension homomorphisms.

We shall use that the category $\TA$ of topological Abelian groups has {both coproducts and pushouts}. We refer the reader to  \cite{HIGGINS1977152,NICKOLAS2002403} for a description of the co\-pro\-duct topology.  Given an extension $E=(\xymatrix{  B \ar[r]^\alpha& X \ar[r]^\beta & A )}$ and continuous homomorphism $u \colon B \to B'$, there is a commutative diagram
$$\xymatrix{  \ar[d]_{u} B \ar[r]^\alpha &  \ar[d]^{\theta}  X \ar[r]^\beta & A   \ar[d]^{\id_A}\\
    B' \ar[r]^{\alpha'}& X' \ar[r]^{\beta' } & A. }$$
  such that the left square is a pushout. We denote by $u_*(E)$ the extension $E'=(\xymatrix{  B' \ar[r]^{\alpha'}& X' \ar[r]^{\beta'} & A })$.

Dually, the category $\TA$ of topological Abelian groups also has  products and pullbacks, and
given an extension $E=(\xymatrix{  B \ar[r]^\alpha& X \ar[r]^\beta & A })$ and a continuous homomorphism $v \colon A' \to A$, there is a commutative diagram
$$\xymatrix{  \ar[d]_{\id_B} B \ar[r]^{\alpha'} &  \ar[d]^{\theta}  X' \ar[r]^{\beta'} & A'   \ar[d]^{v}\\
    B \ar[r]^{\alpha}& X \ar[r]^{\beta } & A. }$$
  such that the right square is a pullback. We denote by $v^*(E)$ the extension $E'=\xymatrix{  B \ar[r]^{\alpha'}& X' \ar[r]^{\beta'} & A' }$.

  The Baer sum of two extensions is given as follows: Given two extensions  $E=(\xymatrix{  B \ar[r]^\alpha& X \ar[r]^\beta & A })$ and  $E'=(\xymatrix{  B \ar[r]^{\alpha'}& X' \ar[r]^{\beta'} & A })$, denote by $E \times E'$ the product extension $\xymatrix{  B \times B \ar[r]^{\alpha \times \alpha'}& X \times X' \ar[r]^{\beta \times \beta'} & A \times A}.$
Now, take the co-diagonal map $\nabla^B \colon B \times B \to B, \nabla^B (b, b')=b+b'$, and consider the extension
$ (\nabla^B_*)(E \times E')=(B \to X'' \to A \times A)$.  Afterwards, take   the diagonal map $\Delta_A\colon A \to A \times A, \Delta_A(a)=(a,a)$, and consider the extension $\Delta_A^*( (\nabla^B_*)(E \times E'))=(B \to X''' \to A)$. Define the Baer sum by  $$[E]+[E']=[\Delta_A^*( (\nabla^B_*)(E \times E'))].$$

The category $\mathcal{E}_{\TA} (A, B)$ has a skeleton which is a set, and $E_{\TA}(A, B)$ has the structure of an Abelian group with the Baer sum. We refer the reader to \cite{MacLane95,Fuchs2015}, and for more details on topological Abelian groups, to \cite{Hugo}.  Moreover, the Baer sum of two $ws$-extensions is a  $ws$-extension and $E_{\TA}^{ws}(A, B)$ is a subgroup in $E_{\TA}(A, B)$.

\section{New descriptions of $E_{\TA}^{ws}(A,B)$}

In this section, we show that the category $\mathcal{E}_{\TA}^{ws}(A,B)$ has a skeleton which is a set. We construct two Abelian groups $\Sigma_{A,B}/\sim$ and $\mathcal{Z}_c(A,B)/\mathcal{B}_c(A,B)$, which are isomorphic to $E_{\TA}^{ws}(A,B)$.

Let $A,B \in \TA$. Consider the canonical split extension
\[
\xymatrix{B \ar[r]^{i_1} & B \times A \ar[r]^{p_2} & A,}
\]
where $i_1(b)=(b,0)$ and $p_2(b,a)=a$, and in this case the canonical section $s\colon A \to B \times A$, $s(a)=(0,a)$ is a continuous homomorphism.

Denote by $+$ the continuous sum (for the groups $B, B \times A, A$) and consider the set $\Sigma_{A,B}$ of all possible continuous sums on $B \times A$:
\[
+'_{B \times A} \colon (B\times A) \times (B\times A) \to B\times A,
\]
such that $((B\times A),+'_{B \times A})$ is a topological Abelian group, and $i_1 \colon (B, +) \to (B\times A, +')$, $p_2\colon (B\times A, +') \to (A, +)$ are continuous homomorphisms in $\TA$.

Given $+', +'' \in \Sigma_{A,B}$, we say that $+' \sim +''$ if there is a topological isomorphism $\gamma \colon (B \times A, +') \to (B \times A, +'')$ such that for $b \in B$, $\gamma i_1(b)=i_1(b)$ and for $a\in A$, $p_2\gamma(b,a)=a$. Note that $\gamma \in \mathrm{Aut}_{\pT}(B \times A) \cap \mathrm{Aut}_{\A}(B \times A)$.

We consider the quotient set
\[
\frac{\Sigma_{A,B}}{\sim}
\]
of all such structures up to topological isomorphism. For each $+'_{B \times A}$, we have a canonical extension
\[
E_{+'_{B \times A}}=
\left(\xymatrix{ (B, +) \ar[r]^-{i_1}& (B \times A,+'_{B \times A}) \ar[r]^-{p_2} & (A,+)}\right).
\]
Therefore, there is a canonical inclusion functor $\Sigma_{A,B} \to \mathcal{E}^{ws}(A,B)$ and a canonical map
\[
\bar{E} \colon \frac{\Sigma_{A,B}}{\sim} \to E^{ws}(A,B), 
\quad [+'_{B \times A}] \mapsto [E_{+'_{B \times A}}].
\]

\begin{theorem}\label{sumproduct}
For any topological Abelian groups $A, B$, the canonical map
\[
\bar{E} \colon \frac{\Sigma_{A,B}}{\sim} \to E_{\TA}^{ws}(A,B)
\]
is a pointed bijection. Moreover, there is an Abelian group structure on $\frac{\Sigma_{A,B}}{\sim}$ induced by $\bar{E}$ such that $\bar{E}$ is an isomorphism of Abelian groups.
\end{theorem}

\begin{proof}

By the definition of the relation $\sim$,  {the map} $\frac{\Sigma_{A,B}}{\sim} \to E_{\TA}^{ws}(A,B)$ is injective.

 To prove that it is a surjective map, assume that  $E=(\xymatrix{B \ar[r]^\alpha & X\ar[r]^\beta & A})$  is an $ws$-extension of $B$ by $A$.
  By Lemma \ref{sectionretraction} and Remark \ref{homeomorphism}, there is  a  pointed
continuous homeomorphism
$\boldsymbol{\phi} \colon B\times A \to X$
such that the following diagram is commutative:

 $$\xymatrix{ \ar[d]_{\id_B}  B \ar[r]^{i_1}& \ar[d]_{\boldsymbol{\phi}}   B \times A \ar[r]^{p_2} & A \ar[d]^{\id_A} \\
 B \ar[r]^{\alpha} &  X \ar[r]^{\beta} & A.  }$$

Since $\boldsymbol{\phi} $ is a homeomorphism, there is a unique continuous sum $+'$ on $B \times A$ such that
$\boldsymbol{\phi} \colon (B\times A, +') \to (X,+)$ is a topological isomorphism. Moreover,
\begin{center}
$i_1(b+b')=  \boldsymbol{\phi}^{-1}\alpha(b+b')=\boldsymbol{\phi}^{-1}(\alpha(b)+\alpha(b'))= (b,0) +' (b',0),$

$  p_2((b,a)+' (b',a'))=  \beta  \boldsymbol{\phi}(((b,a)+' (b',a')))= \beta \boldsymbol{\phi}((b,a))+ \beta\boldsymbol{\phi}( (b',a'))=a+a'.$
\end{center}
 \end{proof}

Following \cite{Fuchs2015}, an alternative way of studying  $E_{\A}(A,B)$ and $E_{\TA}^{ws}(A,B)$  consists of taking
$$\mathcal{Z}(A, B) = \{ f \colon A \times A \to B\ |\ f\ \text{is a pointed map satisfying the Eqs } \eqref{com}~\text{and}~\eqref{aso} \}$$
 \begin{equation} \label{com}
f(a, a')=f(a',a),\ a, a' \in A
\end{equation}
\begin{equation} \label{aso}
f(a,a')+f(a+a',a'')=f(a,a'+a'')+f(a',a''),\ a, a', a'' \in A,
\end{equation}
$$\mathcal{Z}_c(A, B) = \{ f \colon A \times A \to B\ |\ f \in Z(A, B)  \text{ and }  f \text{ is a pointed continuous  map}\},$$
where $(0,0)$ is taken as the  base point of $A\times A$, and $0$ is the base point of $B$.

For each pointed  $h \colon A \to B$, define $D(h) \colon A \times A \to B$ by
\begin{center}
$D(h)(a,a')=h(a)+h(a')-h(a+a'), \quad (a,a') \in A \times A.$
\end{center}
Note that  $D(h) \in \mathcal{Z}(A, B)$, and if $h$ is continuous, then $D(h) \in \mathcal{Z}_c(A, B).$

We also can take
\begin{center}
$\mathcal{B}(A, B)=\{ f \in  \mathcal{Z}(A, B)\ | \text{ there is a pointed map } h \colon A \to B \text{ such that } f=D(h)\}$ and

$\mathcal{B}_c(A, B)=\{ f \in  \mathcal{Z}_c(A, B)\ | \text{ there is a  pointed continuous  map } h \colon A \to B \text{ such that } f=D(h) \}.$
\end{center}
The sum of the Abelian group  $B$ induces canonical  structures of Abelian groups in $\mathcal{Z}(A, B)/\mathcal{B}(A, B)$ and in $\mathcal{Z}_c(A, B)/\mathcal{B}_c(A, B).$

 For a given a $ws$-extension  $E=(\xymatrix{B \ar[r]^\alpha & X\ar[r]^\beta & A})$ of $B$ by $A$, we pick a continuous pointed  section ${\bf{s}} \colon A \to X$ that assigns an element of $X$ in the coset corresponding to $a$;
that is ${\bf{s}}(a)\in \beta^{-1}(a)$.
 By Lemma \ref{sectionretraction} and Remark \ref{homeomorphism}, there is  a
unique continuous pointed retraction ${\bf{r}}\colon X \to B$ (${\bf{r}} j=\id_B$ in $\pT$, ${\bf{r}}(0)=0$), such that ${\bf{r}}(x+\alpha(b))={\bf{r}}(x)+b$ for every $x \in X$,  $b \in B$,  and $\alpha{\bf{r}}+{\bf{s}}\beta=\id_X.$

  Now, we can define the pointed  continuous map $f_E=f$,
 $f\colon A \times A \to B,$
 $ f(a, a') = {\bf{r}} ( {\bf{s}}(a)+  {\bf{s}}(a')- {\bf{s}}(a+a')).$
 \medskip

 \begin{theorem}\label{cocycles} The map $E_{\TA}^{ws}(A,B) \to \mathcal{Z}_c(A, B)/\mathcal{B}_c(A,B)$, $[E]\to [f_E]$ is well defined, and it is  an  isomorphism of Abelian groups:
\begin{center}
$E_{\TA}^{ws}(A,B) \cong \mathcal{Z}_c(A, B)/\mathcal{B}_c(A,B),$
\end{center}
 where $E_{\TA}^{ws}(A,B)$ is provided with the Baer sum, and the sum in $\mathcal{Z}_c(A, B)/\mathcal{B}_c(A,B)$ is induced by the sum in $B$.
\end{theorem}

\begin{proof}

For a given a $ws$-extension $E=(\xymatrix{B \ar[r]^\alpha & X\ar[r]^\beta & A})$ of $B$ by $A$ {{and a continuous pointed section ${\bf{s}}\colon A \to X$}},
consider  the  pointed  continuous map $f_E=f$:
\begin{center}
$f\colon A \times A \to B
 , \quad  f(a, a') = {\bf{r}} ( {\bf{s}}(a)+  {\bf{s}}(a')- {\bf{s}}(a+a')),$
\end{center}
{where ${\bf{r}} \colon X \to B$ is the continuous pointed retraction associated to ${\bf{s}}$; see (ii) in Remark \ref{homeomorphism}. }

 Note that for $a, a' \in A$, we have
\begin{center}
 ${\bf{s}}(a)+{\bf{s}}(a')-{\bf{s}}(a+a')= (\alpha{\bf{r}}+{\bf{s}}\beta) ({\bf{s}}(a)+{\bf{s}}(a')-{\bf{s}}(a+a'))= \alpha({\bf{r}} ({\bf{s}}(a)+{\bf{s}}(a')-{\bf{s}}(a+a')))= \alpha f(a,a'),$
\\
$  f(a, a') = {\bf{r}} ( {\bf{s}}(a)+  {\bf{s}}(a')- {\bf{s}}(a+a'))= {\bf{r}} ( {\bf{s}}(a')+  {\bf{s}}(a)- {\bf{s}}(a'+a))=f(a',a).$
\end{center}

And for $a, a', a'' \in A$,
$ {\bf{s}}(a+a')+  {\bf{s}}(a'')- {\bf{s}}((a+a')+a'')= \alpha (f(a+a',a'')),$
${\bf{s}}(a)+  {\bf{s}}(a'+a'')- {\bf{s}}(a+(a'+a''))= \alpha (f(a, a'+a'')),$
$ \alpha f(a,a')+ \alpha f(a+a',a'')= {\bf{s}}(a)+{\bf{s}}(a')-{\bf{s}}(a+a')+ {\bf{s}}(a+a')+  {\bf{s}}(a'')- {\bf{s}}((a+a')+a'')= {\bf{s}}(a)+{\bf{s}}(a')+  {\bf{s}}(a'')- {\bf{s}}(a+(a'+a''))=   {\bf{s}}(a)+  {\bf{s}}(a'+a'')- {\bf{s}}(a+(a'+a''))+ {\bf{s}}(a')+  {\bf{s}}(a'')- {\bf{s}}(a'+a''))= \alpha f(a,a'+a'')+ \alpha f(a',a'').$

Since $\alpha$ is a monomorphism, we have
\begin{center}
$  f(a,a')+ f(a+a',a'')=  f(a,a'+a'')+  f(a',a'').$
\end{center}
Therefore, the pointed  continuous map
$f\colon A \times A \to B$ satisfies the Eqs \eqref{com} and \eqref{aso}. This implies that $f \in \mathcal{Z}_c(A,B).$

Note that for two pointed sections ${\bf{s_1}}, {\bf{s_2}}$ of $\beta \colon X \to A$, there is a unique  pointed continuous map $h\colon A \to B$ such that
${\bf{s_1}}(a)-{\bf{s_2}}(a)={\alpha(h(a))}$ for every $a \in A$. Then, for the maps $f_1, f_2$ associated to  ${\bf{s_1}}, {\bf{s_2}}$,  we have
\begin{center}
  ${\bf{s_1}}(a)+{\bf{s_1}}(a')-{\bf{s_1}}(a+a')= \alpha f_1(a,a'),$

   ${\bf{s_2}}(a)+{\bf{s_2}}(a')-{\bf{s_2}}(a+a')= \alpha f_2(a,a'),$

     ${\bf{s_1}}(a)-{\bf{s_2}}(a)+{\bf{s_1}}(a')-{\bf{s_2}}(a')-({\bf{s_1}}(a+a')-{\bf{s_2}}(a+a'))= \alpha (f_1(a,a')-f_2(a,a')).$
\end{center}
 Therefore, $f_1(a,a')-f_2(a,a')=h(a)+h(a')-h(a+a')$. This implies that $f_1-f_2 \in \mathcal{B}_c(A,B).$

Now, for a given $f \in \mathcal{Z}_c(A,B)$, one can define a  new continuous map on $B\times A$:
\begin{center}
$(b,a)+_f (b',a')=(b+b'+f(a,a'), a+a').$
\end{center}
The continuous map $+_f \colon (B\times A) \times  (B\times A) \to B\times A$ satisfies the following properties:

Commutative: Using property \eqref{com}, we have
\begin{center}
$(b,a)+_f (b',a')=(b+b'+f(a,a'), a+a')=(b'+b+f(a',a), a'+a)=(b',a')+_f (b,a)$
\end{center}

Associative:
Since $f$ satisfies \eqref{aso}, we have
$((b,a)+_f (b',a'))+_f(b'', a'')=(b+b'+f(a,a'), a+a')+_f(b'',a'')= (b+b'+f(a,a')+b''+f(a+a',a''), (a+a')+a'')= (b+(b'+b''+f(a',a''))+f(a,a'+a''), a+(a'+a'')=(b,a)+_f((b'+b''+f(a',a''), a'+a'')=(b,a)+_f ( (b',a')+_f(b'', a''))$.

We can also check the properties of the neutral element and inverse element for the continuous sum~$+_f$.

Since $f$ is a pointed map, we have  $f(0,0)=0$. By Eq \eqref{aso}, for $a'=a''=0$, it follows that
$f(a,0)+f(a,0)=f(a,0)+f(0,0)$. Therefore,  we have $f(a,0)=f(0,0)=0=f(0,a)$, where we have also applied Eq \eqref{com}. Now, we have
\begin{center}
$(b,a)+_f (0,0)=(b+0+f(a,0), a+0)=(b,a),$

$(b,a)+_f (-b-f(a,-a),-a)=(b-b-f(a,-a)+f(a,-a), a-a)=(0,0).$
\end{center}
Then, we have that $(B\times A, +_f ) $ has the structure of a topological Abelian group.

Note that the inclusion map $i_1 \colon (B, +) \to (B \times A, +_f)$ is a continuous  homomorphism $i_1(b+b')=(b+b',0)=(b+b'+f(0,0), 0+0)=(b,0)+_f (b',0),$
and the projection $p_2 \colon (B \times A, +_f) \to (A,+)$ is a continuous homomorphism
$p_2((b,a)+_f (b',a'))=p_2((b+b'+f(a,a'), a+a')) = a+a'.$

We also have that $i_2 \colon A \to B \times A$, $i_2(a)= (0,a)$ is a continuous section.

This implies that $E_f=(\xymatrix{  (B, +) \ar[r]^-{i_1}& (B \times A,+'_{B \times A}) \ar[r]^-{p_2} & (A,+) })$ is a $ws$-extension.

In order to check that $E$ is equivalent to $E_{E_f}$, we have that  the
continuous map
$\boldsymbol{\phi} \colon B\times A \to X,  \boldsymbol{\phi}(b,a)= \alpha(b)+ {\bf{s}}(a)$ verifies
\begin{center}
$\beta \phi (b,a)= \beta (\alpha(b)+{\bf{s}}(a))=a,$

$\phi i_1(b)=\phi (b,0)= \alpha(b),$

$\phi i_2(a) =\phi (0,a)=0+{\bf{s}}(a) = {\bf{s}}(a).$
\end{center}

These {equations} imply that the following diagram commutes and the pointed section ${\bf{s}}$ corresponds to the pointed section $i_2$.

 $$\xymatrix{ \ar[d]_{\id_B}  B \ar[r]^{i_1}& \ar[d]_{\boldsymbol{\phi}}   B \times A \ar[r]^{p_2} & A \ar[d]^{\id_A} \\
 B \ar[r]^{\alpha} &  X \ar[r]^{\beta} & A.  }$$

We also have that
$\phi ((b,a)+_f (b',a'))=\phi (b+b'+f(a,a'), a+a')= \alpha(b+b'+f(a,a'))+ {\bf{s}}(a+a')$
$= \alpha(b)+ {\bf{s}}(a)+\alpha(b')+ {\bf{s}}(a)=\phi (b,a)+ \phi (b',a')$,
and  $\phi \colon (B \times A, +_f) \to (X,+) $ is an topological isomorphism  in $\TA$.

Now given $f\in \mathcal{Z}_c(A,B)$, we check that $f_{E_f}=f$.

From the equations
$(i_2(a) +_f i_2(a')-_f i_2(a+a')=i_1(f_{E_f} (a,a'))=(f_{E_f} (a,a'),0), $
$(i_2(a) +_f i_2(a'))-_f i_2(a+a')= ((0,a) +_f (0,a'))-_f (0,a+a')=((f(a,a'), a+a') -_f (0,a+a'))=(f(a,a'), a+a')+_f(-f(a+a', -(a+a')), -(a+a')))= (f(a,a')-f(a+a', -(a+a'))+f(a+a',-(a+a')), 0))=((f(a,a'),0).$

It follows that $f(a.a')=f_{E_f}(a,a')$ for every $a, a' \in A$.

Now the proof is completed by verifying that the map
\[
E_{\mathsf{TA}}^{ws}(A,B)\;\longrightarrow\;
\mathcal{Z}_c(A,B)/\mathcal{B}_c(A,B),
\qquad [E] \longmapsto [f_E],
\]
is a homomorphism of Abelian groups (for the algebraic case, see
\cite[Chapter~9]{Fuchs2015}).

Suppose that \(E\) and \(E'\) are \(ws\)-extensions of \(B\) by \(A\).
Their Baer sum is given by
\[
E + E' \;=\; \Delta_A^{*}\bigl((\nabla^{B})_{*}(E \times E')\bigr).
\]
Let \(f_E, f_{E'} \colon A \times A \to B\) be the cocycles associated to
the \(ws\)-extensions \(E\) and \(E'\), respectively.
It is clear that
\[
(f_E \times f_{E'} )\bigl((a_1,a_2),(a_1',a_2')\bigr)
   = \bigl(f_E(a_1,a_2),\, f_{E'}(a_1',a_2')\bigr)
\]
is the cocycle associated to \(E \times E'\).
The pointed continuous map
\[
((a_1,a_2),(a_1',a_2')) \longmapsto
   f_E(a_1,a_2) + f_{E'}(a_1',a_2')
\]
is the cocycle associated to \((\nabla^{B})_{*}(E \times E')\), and
\[
(f_E + f_{E'})(a_1,a_2)
   = f_E(a_1,a_2) + f_{E'}(a_1,a_2)
\]
is the cocycle associated to
\(\Delta_A^{*}\bigl((\nabla^{B})_{*}(E \times E')\bigr)\).
Therefore,
\[
f_{[E] + [E']} \;=\; [\,f_E + f_{E'}\,].
\]
\end{proof}

\section{Comparing $ws$-extensions and algebraic extensions}

Suppose that $A, B \in \TA$. Consider on the Hom-sets
$\mathrm{Hom}_{\pS}(A,B)$ and $\mathrm{Hom}_{\pT}(A,B)$ that
the Abelian group structure is induced by the addition in $B$,
and let us examine the following commutative diagram of Abelian groups:
\[
\xymatrix{
\mathrm{Hom}_{\pT}(A,B) \ar[r]^{in} \ar[d]_{D_c} 
& \mathrm{Hom}_{\pS}(A,B) \ar[r] \ar[d]_{D}
& \mathrm{Hom}_{\pS}(A,B)/\mathrm{Hom}_{\pT}(A,B) \ar[d]^{\bar{D}} \\
\mathcal{Z}_c(A,B) \ar[r]^{in}
& \mathcal{Z}(A,B) \ar[r]
& \mathcal{Z}(A,B)/\mathcal{Z}_c(A,B)
}
\]
where
\[
D_c(h)(a,a')=h(a)+h(a')-h(a+a'), \quad h \in \mathrm{Hom}_{\pT}(A,B),
\]
\[
D(h)(a,a')=h(a)+h(a')-h(a+a'), \quad h \in \mathrm{Hom}_{\pS}(A,B),
\]
\[
\bar{D}([h])=[D(h)], \quad h \in \mathrm{Hom}_{\pS}(A,B),
\]
and $in$ denotes the canonical inclusions.

\begin{definition}
For $A, B \in \TA$ and $i \in \{0,1\}$, consider the Abelian groups
\[
\overline{\mathrm{ex}}^0(A,B)=\mathrm{Ker}(\bar{D}), 
\quad
\overline{\mathrm{ex}}^1(A,B)=\mathrm{Coker}(\bar{D}).
\]
\end{definition}

\begin{proposition}\label{diagram}
Let $A, B$ be topological Abelian groups in $\TA$. Then, the following diagram in $\A$ is commutative:
\[
\xymatrix{
\ar[d]\mathrm{Hom}_{\TA}(A,B) \ar[r]^{i_*^0}
& \ar[d]   \mathrm{Hom}_{\A}(A,B) \ar[r]^{p_*^0}
&  \ar[d] \overline{\mathrm{ex}}^0(A,B) \\
\mathrm{Hom}_{\pT}(A,B)  \ar[r]^{in} \ar[d]_{D_c}
& \mathrm{Hom}_{\pS}(A,B) \ar[r] \ar[d]_{D}
& \mathrm{Hom}_{\pS}(A,B)/\mathrm{Hom}_{\pT}(A,B) \ar[d]^{\bar{D}} \\
\mathcal{Z}_c(A,B) \ar[r]^{in} \ar[d]
& \mathcal{Z}(A,B) \ar[r] \ar[d]
& \mathcal{Z}(A,B)/\mathcal{Z}_c(A,B) \ar[d] \\
E_{\TA}^{ws}(A,B) \ar[r]^{i_*^1}
& E_{\A}(A,B) \ar[r]^{p_*^1}
& \overline{\mathrm{ex}}^1(A,B)
}
\]
and we also have:
\begin{itemize}
\item[(i)] $\mathrm{Ker}(D_c)=\mathrm{Hom}_{\TA}(A,B)$,
\item[(ii)] $\mathrm{Ker}(D)=\mathrm{Hom}_{\A}(A,B)$,
\item[(iii)] $\overline{\mathrm{ex}}^0(A,B)\cong \{[h] \in \mathrm{Hom}_{\pS}(A,B)/\mathrm{Hom}_{\pT}(A,B)\mid D(h) \in \mathcal{Z}_c(A,B)\}$,
\item[(iv)] $\mathrm{Coker}(D_c)\cong E_{\TA}^{ws}(A,B)$,
\item[(v)] $\mathrm{Coker}(D)\cong E_{\A}(A,B)$,
\item[(vi)] $\overline{\mathrm{ex}}^1(A,B) \cong \mathcal{Z}(A,B)/( \mathcal{Z}_c(A,B)+\mathcal{B}(A,B))$.
\end{itemize}

The homomorphisms in the diagram are given by
$i_*^0(h)=h$, $h \in \mathrm{Hom}_{\TA}(A,B)$,
and $p_*^0(h')=[h']$, $h' \in \mathrm{Hom}_{\A}(A,B)$.
For a $ws$-extension $E$, we have $i_*^1([E])=[E]$, and for an algebraic extension $E'$, $p_*^1(E')=[f_{E'}]$.
\end{proposition}

\begin{proof}
By Theorem~\ref{cocycles}, we have $\mathrm{Coker}(D_c)\cong E_{\TA}^{ws}(A,B)$, and similarly $\mathrm{Coker}(D)\cong E_{\A}(A,B)$. The rest of the proof is a routine verification.
\end{proof}

\begin{theorem}\label{sixterms}
Given $A, B \in \TA$, there is a connecting homomorphism
\[
\delta \colon \overline{\mathrm{ex}}^0(A,B) \to E_{\TA}^{ws}(A,B),
\]
given by $\delta([h])=[E_{in^{-1}(D(h))}]$, such that the following sequence is exact:

\[
\xymatrix{
0 \ar[r] 
& \mathrm{Hom}_{\TA}(A,B) \ar[r]^{i_*^0}
& \mathrm{Hom}_{\A}(A,B) \ar[r]^{p_*^0}
& \overline{\mathrm{ex}}^0(A,B) \ar[r]^-{\delta}
&
\\
& E_{\TA}^{ws}(A,B) \ar[r]^{i_*^1}
& E_{\A}(A,B) \ar[r]^{p_*^1}
& \overline{\mathrm{ex}}^1(A,B) \ar[r]
& 0.
}
\]

\begin{proof}
It is a consequence of Proposition~\ref{diagram} and the Snake Lemma; see \cite[Lemma 2.6 in Chapter I]{Fuchs2015} or \cite[Lemma 5.1 in Chapter III]{Hilton-Stammbach}.
\end{proof}
\end{theorem}

\section{Examples}\label{examples}

{In this section, we analyze the properties  of the Abelian group of  $ws$-extensions in the case of topological Abelian groups with trivial and discrete topologies. We also give examples of $ws$-extensions induced by covering extensions and others related with Bohr topologies.}

\subsection{Trivial topologies and $ws$-extensions }

Given a topological  Abelian group $A$, we can consider the   subgroup  $N_0(A)$  given as follows:\medskip

\begin{lemma}\label{trivialtop} Let $A$ be a topological Abelian group, and consider $$N_0(A)=\bigcap \{U| U \text{ is a neighbourhood  at } 0\}.$$
Then, $N_0(A)$ is a subgroup of $A$,  and it has  trivial topology. Moreover,
\begin{itemize}
  \item[(i)] $N_0 \colon \TA \to \TA$ is a functor,
  \item [(ii)] if $A, B \in \TA$ and $A$  has the trivial topology, then $$\mathrm{Hom}_{\TA}(A, B) \cong  \mathrm{Hom}_{\TA}(A, N_0(B))
  \cong  \mathrm{Hom}_{\A}(A, N_0(B)),$$
    \item [(iii)] and if $A, B \in \TA$ and $A$  has the trivial topology, then $$\mathrm{Hom}_{\pT}(A, B) \cong  \mathrm{Hom}_{\pT}(A, N_0(B)).$$
\end{itemize}
\end{lemma}

\begin{proof} Note that if $U$ is an open neighbourhood at $0$ in $B$ and $f \colon A \to B$ is a continuous map, then $A\subset f^{-1}(U)$. Using this fact,  the proof of (i), (ii),  and (iii) is  routine.
\end{proof}

\begin{remark}
The subgroup $N_0(G)$ can be defined for non-Abelian groups as well. It coincides with the closure of the trivial subgroup $\{e\}$, so it is a closed normal subgroup of $G$ (so, the smallest one).
\end{remark}

\begin{theorem} If $A$ is a topological Abelian group with the trivial topology, then $${E}_{\TA}^{ws}(A, B)\cong  {E}_{\A}(A, N_0(B)),$$
\end{theorem}

\begin{proof}

As a consequence of Lemma \ref{trivialtop}, we have
$\mathrm{Hom}_{\pT}(A, B)
  \cong  \mathrm{Hom}_{\pT}(A, N_0(B)),$ and
  $\mathrm{Hom}_{\pT}(A \times A, B)
  \cong  \mathrm{Hom}_{\pT}(A \times A, N_0(B)).$

  This implies that
  $\mathcal{B}_c(A, B)
  \cong  \mathcal{B}(A, N_0(B))$ and
  $\mathcal{Z}_c(A, B)
  \cong  \mathcal{Z}(A, N_0(B)).$ Then,
  $ \mathcal{Z}_c(A,B)/\mathcal{B}_c(A,B) \cong  \mathcal{Z}(A,N_0(B))/\mathcal{B}(A,N_0(B))$. Now, we can apply Theorem \ref{cocycles} to obtain
$${E}_{\TA}^{ws}(A, B)\cong {E}_{\A}(A, N_0(B)).$$
\end{proof}
\bigskip

\begin{corollary} \label{trivial} Suppose that $A, B \in \TA$. If $A$ has the trivial topology and $B$ is Hausdorff,  then $$E_{\TA}^{ws}(A, B)\cong  0.$$
\end{corollary}

\begin{example}\label{ej1} Consider the extension
$E=(\mathbb{Q}\to\mathbb{R}\to \mathbb{R}/\mathbb{Q}) \in \mathcal{E}_{TA}(\mathbb{R}/\mathbb{Q}, \mathbb{Q}) $
where $\mathbb{Q}, \mathbb{R}$ have the usual topology and   $\mathbb{R}/\mathbb{Q}$ has  the trivial topology. Then,  by Corollary \ref{trivial},
$$E_{\TA}^{ws}(\mathbb{R}/\mathbb{Q},\mathbb{Q})\cong 0.$$

By Theorem \ref{sixterms}, {we have} the long exact sequence
$$\xymatrix{0 \ar[r] &{\rm{Hom}}_{\mathbf{TA}}(\mathbb{R}/\mathbb{Q}, \mathbb{Q})\ar[r] &{\rm{Hom}}_{\mathbf{A}}(\mathbb{R}/\mathbb{Q},\mathbb{Q}) \ar[r] & \overline{{\rm{ex}}}_{\mathbf{TA}}^{0}(\mathbb{R}/\mathbb{Q},\mathbb{Q}) \ar[r] &}$$
      $$\xymatrix{ E_{\mathbf{TA}}^{ws}(\mathbb{R}/\mathbb{Q}, \mathbb{Q}) \ar[r]& E_{\mathbf{A}}(\mathbb{R}/\mathbb{Q}, \mathbb{Q}) \ar[r]& \overline{{\rm{ex}}}_{\mathbf{TA}}^{1}(\mathbb{R}/\mathbb{Q},\mathbb{Q}) \ar[r]& 0 .}$$

  Since $\mathbb{Q}$ is totally disconnected and $\mathbb{R}$ is connected, we have
 ${\rm{Hom}}_{\mathbf{TA}}(\mathbb{R}/\mathbb{Q}, \mathbb{Q})\cong 0$.
Taking into account that $\mathbb{Q}$ is divisible, by \cite[Theorem 2.6 in Chapter IV]{Fuchs2015} (see also \cite[page 267]{Fuchs2015}), we  have that $\mathbb{Q}$ is injective. This fact implies that
       $E_{\mathbf{A}}^{}(\mathbb{R}/\mathbb{Q}, \mathbb{Q}) \cong 0.$

Using the
exactness
of the  sequence above, we have
$$ \overline{{\rm{ex}}}_{\mathbf{TA}}^{0}(\mathbb{R}/\mathbb{Q},\mathbb{Q})  \cong{\rm{Hom}}_{\mathbf{A}}(\mathbb{R}/\mathbb{Q},\mathbb{Q}), $$
$$ \overline{{\rm{ex}}}_{\mathbf{TA}}^{1}(\mathbb{R}/\mathbb{Q},\mathbb{Q})  \cong 0.$$

Since $\mathbb{R}/\mathbb{Q}\cong \bigoplus_{\frak c}\, \mathbb{Q}=: \Q^{(\frak c)}$, we also  have that
$$\overline{{\rm{ex}}}_{\mathbf{TA}}^{0}(\mathbb{R}/\mathbb{Q},\mathbb{Q})\cong \prod_{\frak c} \mathbb{Q}=: \Q^\frak c.$$
\end{example}

\subsection{Discrete topologies and $ws$-extensions}

We examine in this subsection the Abelian group of $ws$-extensions under the conditions that either the quotient is discrete or the subspace is discrete.

Firstly, we study the case of weak-split extensions:
$$E=(\xymatrix{ B \ar[r]^j& X \ar[r]^p& D}),$$
where $D$ is  a discrete topological Abelian group.

Since any pointed section ${\boldsymbol{\sigma}}\colon D \to X$ is continuous, we have that the canonical map $E_{\TA}^{ws}(D,B) \to E_{\TA}(D,B)$ is surjective,  and it is easy to check that in this case   $E_{\TA}^{ws}(D,B) \cong E_{\TA}(D,B).$

 For the case of $ws$-extensions of $B$ by $D$ with $D$ discrete, we have the following  properties:\medskip

 \begin{proposition} Let $B, D$ be topological Abelian groups and suppose that $D$ is discrete. Then,

 \begin{itemize}
  \item[(i)]
  $\mathrm{Hom}_{\mathbf{TA}}(D,B) \cong \mathrm{Hom}_{\mathbf{A}}(D,B)$,
  \item[(ii)]    $E_{\mathbf{TA}}^{ws}(D,B)\cong E_{\mathbf{A}}(D,B),$
  \item[(iii)] $\overline{\mathrm{ex}}^0_{\mathbf{TA}}(D,B)\cong 0$,
    \item[(iv)] $\overline{\mathrm{ex}}^1_{\mathbf{TA}}(D,B)\cong 0.$
\end{itemize}
  \end{proposition}

  \begin{proof}

  Since $D$ is discrete, we have the following:

$\mathrm{Hom}_{\pT}(D, B)
  \cong  \mathrm{Hom}_{\pS}(D, B),$ \quad
  $\mathrm{Hom}_{\pT}(D \times D , B)
  \cong  \mathrm{Hom}_{\pS}(D \times D , B).$

  This implies that
  $\mathcal{B}_c(D, B)
  \cong  \mathcal{B}(D, B)$ and
  $\mathcal{Z}_c(D, B)
  \cong  \mathcal{Z}(D, B).$

  Then,
  $ \mathcal{Z}_c(A,B)/\mathcal{B}_c(D,B) \cong  \mathcal{Z}(D,B)/\mathcal{B}(D,B)$. Now, we can apply Theorem \ref{cocycles} to obtain
$${E}_{\TA}^{ws}(D, B)\cong {E}_{\A}(D, B).$$
By Theorem \ref{sixterms} the following sequence is exact:
   $$0 \to {\rm{Hom}}_{\mathbf{TA}}(D, B)\to {\rm{Hom}}_{\mathbf{A}}(D,B) \to \overline{{\rm{ex}}}_{\mathbf{TA}}^{0}(D,B) \to  E_{\mathbf{TA}}^{ws}(D, B) \to $$
      $$E_{\mathbf{A}}(D, B)\to \overline{{\rm{ex}}}_{\mathbf{TA}}^{1}(D,B)\to 0. $$

 Since $ {\rm{Hom}}_{\mathbf{TA}}(D, B)\to {\rm{Hom}}_{\mathbf{A}}(D,B)$ is also an isomorphism, we obtain
$\overline{\mathrm{ex}}^0_{\mathbf{TA}}(D,B)\cong 0$, and
 $\overline{\mathrm{ex}}^1_{\mathbf{TA}}(D,B)\cong 0.$
\end{proof}

  Next, we consider the case when the subgroup of the $ws$-extesnsion is discrete, although beforehand we introduce the concept of a covering extension.

  \begin{definition} An extension  of  topological Abelian groups   $\xymatrix{B\ar[r]^{j} & X \ar[r]^{p} & A }$ is said to be a \emph{covering extension} if there is an open $U$ in $A$ at $0$, and disjoint open subsets $U_b$ in $X$, $b\in B$ such that $j(b) \in U_b$, $U_{b_1} \cap U_{b_2}=\emptyset$, for $b_1 \not = b_2$,  $p^{-1} (U) = \bigcup_{b \in B}  U_b$, and $p|_{U_b} \colon U_b \to U$ is a homeomorphism.
 If there is an open neighbourhood $U$ in $A$ with these properties,  it is said that  $p$ is \emph{trivial over} $U$.
  \end{definition}

  \begin{example}\label{dos} The standard extension $\xymatrix{\mathbb{Z} \ar[r]^{j} & \mathbb{R}  \ar[r]^{p} & \mathbb{R}/ \mathbb{Z}  }$ is a covering extension (note that $\mathbb{R}/ \mathbb{Z}\cong \mathbb{T} $). For each $n \in \mathbb{N}$, we have the extension
 $$\xymatrix{\mathbb{Z}/n \mathbb{Z} \ar[r]^{j_n} & \mathbb{R}/ \mathbb{Z}  \ar[r]^{p_n} & \mathbb{R}/ \mathbb{Z},  }$$ where $p_n(r+\mathbb{Z})= nr +\mathbb{Z}$, is also a covering extension.  In these cases, for any $y \in \mathbb{R}/ \mathbb{Z}$, and $y\not = 0$, the open neighbouhood $U= \mathbb{R}/ \mathbb{Z}\setminus \{y\}$ verifies that $p$ and $p_n$, and $n \in \mathbb{N}$  are trivial over $U$.
 Note that if $a$ is an irrational real number and  $y=a+\mathbb{Z}$, then $\mathbb{Q}/ \mathbb{Z}\subset U$.
 \end{example}

 \begin{remark} We note that if the extension  $\xymatrix{B\ar[r]^{j} & X \ar[r]^{p} & A }$ is a covering extension, then $B$ has the  discrete topology, and if $p$ is trivial over an open neighbourhood  $U$ at $0$, then we can take  the
  open covering $A=\bigcup_{a\in A} a+U$, and
we have that $p$ is also trivial over $a+U$.
Conversely, if $B$ carries the discrete topology, then there exists an open set
$U \subseteq X$ such that $U \cap j(B) = \{0\}$, the restriction of $p$ to $U$ is injective,
and $p(U)$ is open. It follows that $p$ is trivial on $U$, and hence the extension is a
covering~extension.
 \end{remark}

 \begin{proposition}\label{csplit} Suppose the extension   $\xymatrix{B\ar[r]^{j} & X \ar[r]^{p} & H }$ is an extension covering that  is trivial on an open neighbourhood $U$ at $0$, and let $A$ be  a subgroup of $H$ such that $A \subset U$. Then $\xymatrix{B\ar[r]^{j} & p^{-1}(A) \ar[r]^-{p|_{p^{-1}(A)}} & A }$ is a weak-split extension.
  \end{proposition}

  \begin{proof} Since $\theta_0=p|_{U_0} \colon U_0 \to U$ is a homeomorphism, then ${\bf{s}}= \theta_0^{-1}|_A \colon A \to p^{-1}(A)$ is a pointed section of $p|_{p^{-1}(A)}$.
  \end{proof}

 \begin{corollary}\label{Cor2} The extensions $$E=(\xymatrix{\mathbb{Z} \ar[r]^{j} & \mathbb{Q}  \ar[r]^{p} & \mathbb{Q}/ \mathbb{Z} }), \quad E_n=(\xymatrix{\mathbb{Z}/n \mathbb{Z} \ar[r]^{j_n} & \mathbb{Q}/ \mathbb{Z}  \ar[r]^{p_n} & \mathbb{Q}/ \mathbb{Z} }),$$
 are weak-split extensions. Moreover, $ \mathbb{Q} $ is not topologically isomorphic to  $ \mathbb{Z} \times \mathbb{Q}/ \mathbb{Z} $ in $\TA$ and similarly,  $ \mathbb{Q}/ \mathbb{Z}$  is not topologically isomorphic to  $\mathbb{Z}/n \mathbb{Z}\times \mathbb{Q}/ \mathbb{Z} $ in $\TA$ (note that $ \mathbb{Q} $ is homeomorphic to  $ \mathbb{Z} \times \mathbb{Q}/ \mathbb{Z} $ in $\pT$ and similarly,  $ \mathbb{Q}/ \mathbb{Z}$  is homeomorphic  to  $\mathbb{Z}/n \mathbb{Z}\times \mathbb{Q}/ \mathbb{Z} $ in $\pT$).
  \end{corollary}

   \begin{proof} From Proposition \ref{csplit} and Example \ref{dos}, it follows that $E$ and $E_n$ are $ws$-extensions.

    Note  $ \mathbb{Q} $ does not have  nonzero torsion elements, and $\mathbb{Z} \times \mathbb{Q}/ \mathbb{Z} $ has nonzero torsion elements. This implies that $ \mathbb{Q} $ is not algebraically isomorphic to $\mathbb{Z} \times \mathbb{Q}/ \mathbb{Z} $.

    Similarly, the subgroup of $n$-torsion elements of $\mathbb{Q}/ \mathbb{Z}$ is isomorphic to  $\mathbb{Z}/n \mathbb{Z} $, but the subgroup of $n$ -torsion elements of $\mathbb{Z}/n \mathbb{Z} \times \mathbb{Q}/ \mathbb{Z}$ is isomorphic to $\mathbb{Z}/n \mathbb{Z}  \times \mathbb{Z}/n \mathbb{Z} $. Therefore, $ \mathbb{Q}/ \mathbb{Z}$  is not algebraically isomorphic to   $\mathbb{Z}/n \mathbb{Z}\times \mathbb{Q}/ \mathbb{Z} $ in $\pT$
  \end{proof}

  A consequence of the above results is the existence of topological Abelian groups $A, B$ such that $E_{\TA}^{ws}(A,B)$ is not a trivial Abelian group.\medskip

  \begin{corollary} The topological Abelian groups $\mathbb{Q}/ \mathbb{Z}, \, \mathbb{Z},$ and  $\mathbb{Z}/n \mathbb{Z} $ with the usual topologies verify
   $$E_{\TA}^{ws}(\mathbb{Q}/ \mathbb{Z} , \mathbb{Z})\not \cong 0,  \quad E_{\TA}^{ws}(\mathbb{Q}/ \mathbb{Z} , \mathbb{Z}/n \mathbb{Z} )\not \cong 0.$$
   \end{corollary}
 \medskip

In fact, the results above also hold for subgroups of the additive group of rational numbers $\Q$ that contains the group of integers $\Z$.
\medskip

\begin{corollary}\label{Q1}
Let $G$ be a subgroup of the additive group of rational numbers $\Q$ that contains the group of integers $\Z$. Then, the canonical exact sequence
$$E=(\xymatrix{\mathbb{Z} \ar[r]^{j} & G  \ar[r]^{p} & G/ \mathbb{Z}})$$ is a $ws$-extension when the group $G$ is equipped with
the Euclidean topology.
\end{corollary}
\begin{proof}
  The proof is trivial in the  case where $G$ is finitely generated. Indeed, in such a case the quotient group $G/\Z$ is finite. Therefore, we assume
 without loss of generality that $G$ cannot be finitely generated and, as a consequence, that $G$ is dense in $\Q$.
 Define the pointed section ${\bf s}_{(G/\Z)}\colon G/\Z\to G$ as the restriction mapping ${\bf s}_{(G/\Z)}={\bf s}_{(\Q/\Z)\, |(G/\Z)}$,
  where ${{\bf s}}_{(\Q/\Z)}$ is the continuous pointed section whose existence is proven in Corollary \ref{Cor2}.
  We claim that ${\bf s}_{(G/\Z)}$ is well defined. Indeed, take $g\in G$ and let $q={\bf s}_{(G/\Z)}(g+\Z)$.
  Since ${\bf s}_{(G/\Z)}$ is the restriction of ${\bf s}_{(\Q/\Z)}$ to $G/\Z$, it follows that
  $q-g\in \Z$. Therefore,  $g\in q+\Z\subseteq G$. The continuity of ${\bf s}_{(G/\Z)}$ is now a straightforward
  consequence of the continuity of ${\bf s}_{(\Q/\Z)}$.
\end{proof}

\subsection{Bohr topology and $ws$-extensions}\label{Bohrtopology}

  Given an  Abelian group $G$, we denote by $G^\sharp$ the algebraic group $G$ equipped with the initial topology with respect to  the family of  maps $\mathrm{Hom}_{\A}(G, \mathbb{T})$,  which is called the Bohr topology on $G$. Remark
  that $\mathrm{Hom}_{\A}(G^\sharp, \mathbb{T})\cong Hom_{\mathbf{A}}(G, \mathbb{T})$. For more details and properties on the Bohr topology, we refer the reader to
  \cite{van1990maximal}.

    It is interesting to note that we can consider the following extension of groups in   $\A$:
  $\xymatrix{\mathbb{Z} \ar[r]^{j} & \mathbb{Q}  \ar[r]^{p} & \mathbb{Q}/\mathbb{Z}}.$

  Then, we have the induced extensions
$$(\xymatrix{\mathbb{Z}^\sharp \ar[r]^{i^\sharp} & \mathbb{Q}^\sharp \ar[r]^{p^\sharp} & (\mathbb{Q}/ \mathbb{Z})^\sharp} ),$$
  $$(\xymatrix{\mathbb{Z}^\sharp \ar[r]^{i^\sharp} & \mathbb{Q}^\sharp  \ar[r]^{q} & \mathbb{Q}^\sharp/ \mathbb{Z} ^\sharp}).$$

Note that as consequence of \cite[Remark 10 (b)]{COMFORT2001215},  $ (\mathbb{Q}/ \mathbb{Z})^\sharp\cong  \mathbb{Q}^\sharp/ \mathbb{Z} ^\sharp$.
Moreover, as a consequence of \cite[Theorem 24]{COMFORT2001215}, these short exact sequences split weakly. This implies that the topological space  $\mathbb{Q}^\sharp $ is homeomorphic to $\mathbb{Z}^\sharp \times (\mathbb{Q}/ \mathbb{Z})^\sharp $.  Since $\mathbb{Q}^\sharp $ does not have torsion, it follows that the topological Abelian group   $\mathbb{Q}^\sharp $ is not algebraically isomorphic to  $\mathbb{Z}^\sharp \times (\mathbb{Q}/ \mathbb{Z})^\sharp $. Therefore, $\mathbb{Q}^\sharp $ is not isomorphic to $\mathbb{Z}^\sharp \times (\mathbb{Q}/ \mathbb{Z})^\sharp $ in the category $\TA$.

These arguments  prove that
$$E_{\TA}^{ws}((\mathbb{Q}/ \mathbb{Z})^\sharp , \mathbb{Z}^\sharp)\not \cong 0.$$

Moreover, in this case we  can give the following  computation of  $E_{\TA}^{ws}((\mathbb{Q}/ \mathbb{Z})^\sharp , \mathbb{Z}^\sharp)$.\medskip

\begin{example}\label{Ext^ws}
The group $E_{\TA}^{ws}((\mathbb{Q}/ \mathbb{Z})^\sharp , \mathbb{Z}^\sharp)$ is isomorphic to $Hom(\Q/\Z, \Q/\Z)$.
\end{example}
\begin{proof}
From the long exact sequence associtated to $0 \to \mathbb{Z} \to \mathbb{Q} \to \mathbb{Q}/\mathbb{Z} \to 0$ and taking into account that $\mathbb{Q}$ is a divisible free-torsion group and $ \mathbb{Q}/\mathbb{Z}$ is a torsion group, one has the isomorphism
 $$\text{Hom}_{\A}(\Q/\Z, \Q/\Z) \cong \text{Ext}_{{\bf A}}(\mathbb{Q}/ \mathbb{Z}, \mathbb{Z}) \cong \text{E}_{{\bf A}}(\mathbb{Q}/ \mathbb{Z}, \mathbb{Z}).$$
Furthermore, this isomorphism is realized as follows (see \cite[Chapter~9, Theorem~3.5]{Fuchs2015} for the details of this~construction):

Denoting by $\pi\colon\Q\to\Q/\Z$ the canonical quotient map, each
\[
f \in \operatorname{Hom}(\mathbb{Q}/\mathbb{Z}, \mathbb{Q}/\mathbb{Z})
\]
has  the associated extension
\[
0 \longrightarrow \mathbb{Z}
\overset{\iota}{\longrightarrow} E_f
\overset{p}{\longrightarrow} \mathbb{Q}/\mathbb{Z}
\longrightarrow 0,
\]
with
\[
E_f
=
\left\{
(q,x)\in \mathbb{Q}\times(\mathbb{Q}/\mathbb{Z})
\;\middle|\;
\pi(q)=f(x)
\right\},
\]

\medskip

\noindent such that the map
\[
\Phi :
\operatorname{Hom}(\mathbb{Q}/\mathbb{Z}, \mathbb{Q}/\mathbb{Z})
\longrightarrow
\operatorname{Ext_{{\bf A}}}(\mathbb{Q}/\mathbb{Z}, \mathbb{Z}),
\qquad
f \longmapsto [E_f],
\]
is a group isomorphism.

Let ${{\bf s}}_{(\Q/\Z)}\colon (\Q/\Z)^\sharp\to \Q^\sharp$ be the continuous pointed section whose existence is proven in Corollary~\ref{Cor2}.
We define the pointed section ${\bf s}_{E_f}\colon (\Q/\Z)^\sharp\to E_f^\sharp$ by ${\bf s}_{E_f}(x):=(({{\bf s}}_{(\Q/\Z)}\circ f)(x),x)$.
Note  that $(({{\bf s}}_{(\Q/\Z)}\circ f)(x),x)\in E_f$ since $\pi({{\bf s}}_{(\Q/\Z)}(f(x)))=(\pi\circ {{\bf s}}_{(\Q/\Z)})(f(x))=f(x)$. On the other hand
both maps $f$, being a group homomorphism, and ${{\bf s}}_{(\Q/\Z)}$ are $\sharp$-continuous, which implies the $\sharp$-continuity of ${\bf s}_{E_f}$.
This means that the sequence
\[
0 \longrightarrow \mathbb{Z}^\sharp
\overset{\iota}{\longrightarrow} E_f^\sharp
\overset{p}{\longrightarrow} (\mathbb{Q}/\mathbb{Z})^\sharp
\longrightarrow 0
\]
$ws$-splits, which completes the proof since $f$ was arbitrarily chosen in  $$Hom(\Q/\Z, \Q/\Z).$$

We note that an explicit computation of the group $Hom(\Q/\Z, \Q/\Z)$ can be done using Pontryagin duality; see \cite{Fuchs2015} for further information.
\end{proof}

Next, we extend these facts to a wider setting by applying some of the results obtained previously.
We remark that most results in this section can also be derived from \cite{COMFORT2001215}.
First, we need some preliminary~results.
\medskip

 \begin{proposition}\label{prducts}
Let  $$\{ E_i=(\xymatrix{B_i \ar[r]^{j_i} & X_i  \ar[r]^{p_i} & A_i}) \}_{i\in I}$$ be a family of  $ws$-extensions in {\bf TA}. Then,  the sequence
$$E=(\xymatrix{\prod\limits_{i\in I} B_i \ar[r]^{j} & \prod\limits_{i\in I}X_i  \ar[r]^{p} & \prod\limits_{i\in I}A_i}),$$ where $j((b_i)):=(j_i(b_i))$ and $p((x_i)):=(p_i(x_i))$
is a $ws$-extension.
\end{proposition}
\begin{proof}
  It suffices to notice that if ${\bf s}_i\colon A_i\to X_i$ is a continuous pointed section, then  the map
  ${\bf s}((a_i)):=({\bf s}_i(a_i))$ defines a continuous pointed section of
  $\prod\limits_{i\in I} A_i$ into $\prod\limits_{i\in I} X_i$.
\end{proof}

\begin{example} The extension
 $\xymatrix{\mathbb{Z}/ 2\mathbb{Z} \ar[r]^{\alpha} &\mathbb{Z}/ 4 \mathbb{Z} \ar[r]^{\beta} &\mathbb{Z}/ 2\mathbb{Z}}$ (in $\A$),
  where $\alpha(1+ 2\mathbb{Z})=2+4 \mathbb{Z} $, $\beta (1+4 \mathbb{Z} )= 1+2 \mathbb{Z}$
   has a pointed section ${\bf{s}}(1+2 \mathbb{Z} )= 1+4 \mathbb{Z}$ and ${\bf{s}}(0+4 \mathbb{Z} )= 0+2 \mathbb{Z}$.

 This implies that the induced extension
 $$\xymatrix{\prod_{\mathbb{N}}\mathbb{Z}/ 2\mathbb{Z} \ar[r]^{ \prod_{\mathbb{N}}\alpha} & \prod_{\mathbb{N}}\mathbb{Z}/  4\mathbb{Z} \ar[r]^{ \prod_{\mathbb{N}}\beta} & \prod_{\mathbb{N}}\mathbb{Z}/ 2\mathbb{Z}  }$$
also has  a continuous pointed section, and it is therefore a $ws$-extension.
 \end{example}

The arguments drawn on to prove Corollary \ref{Q1} also
apply for the Bohr topology.

\begin{proposition}\label{Q2}
Let $G$ be a subgroup of the additive group of rational numbers $\Q$ that contains the integers $\Z$. Then, the canonical exact sequence
$$E=(\xymatrix{\mathbb{Z}^\sharp \ar[r]^{j} & G^\sharp  \ar[r]^{\hspace{-0.3cm} p} & (G/ \mathbb{Z})^\sharp}\, ),$$ is a $ws$-extension.
\end{proposition}
\begin{proof}
It was established  in   \cite[Theorem 24]{COMFORT2001215} that the pointed section ${\bf s}_{\Q/\Z}$ is continuous with respect to the
(Bohr) $\sharp$-topology.
Since the $\sharp$-topology is inherited by subgroups, it follows that the pointed section ${\bf s}_{(G/\Z)}={\bf s}_{(\Q/\Z)\, |(G/\Z)}$ is also $\sharp$-continuous (note that it was proven in Corollary \ref{Q1} that ${\bf s}_{(G/\Z)}$ is well defined).
\end{proof}
\medskip

We now look at the subgroups of the additive group of real numbers. In order to do it, we need the following definition
(see \cite[Definition 15]{COMFORT2001215}, \cite[Definition 1.3]{Dikranjan2002}).
\medskip

\begin{definition}\label{ccs}
A group $H$ is a \emph{ccs-group} if for any group $G$ containing $H$ as a subgroup, there is a continuous
cross section $\Gamma\colon (G/H)^\sharp\to G^\sharp$ such that $\pi\circ\Gamma = id_{|G/H}$
(with $\pi\colon G^\sharp\to (G/H)^\sharp$ being the natural quotient homomorphism).
\end{definition}
\medskip

It is a known fact that every finitely generated group $H$ is a ccs-group. In particular, the
group of integers is a ccs-group (see \cite[Corollary 27]{COMFORT2001215}, \cite[Corollary 2.3(a)]{Dikranjan2002}).
From these facts, it follows the proposition below. However, we have included a sketch of its proof for the reader's sake.
\medskip

\begin{proposition}\label{R1}
Let $G$ be a subgroup of the additive group of real numbers $\R$ that contains the group of integers $\Z$. Then, the canonical exact sequence
$$E=(\xymatrix{\mathbb{Z}^\sharp \ar[r]^{j} & G^\sharp  \ar[r]^{p} & (G/ \mathbb{Z})^\sharp)}$$ is a $ws$-extension.
\end{proposition}
\begin{proof}
We proceed by a case study on the group $G$.\\
(i) We first assume that $G=\R$. Since $\R\cong  \Q^{(\frak c)}$, we may assume that $\R\cong  \Q\oplus \Q^{(\frak c)}$ with $\Z\subseteq \Q$.
Set $H=\Q^{(\frak c)}$. Then, we have $\R\cong \Q\oplus H$. Therefore,
we have the exact sequence
$$E={\xymatrix{\mathbb{Z}^\sharp \ar[r]^{\hspace{-0.5cm} j} & (\Q\bigoplus H)^\sharp  \ar[r]^{\hspace{-0.5cm} p} & ((\Q/\Z) \bigoplus H)^\sharp})}.$$
Again, using that $(A\bigoplus B)^\sharp\cong A^\sharp\bigoplus B^\sharp$ for every pair of Abelian groups $\{A,B\}$,
the above sequence can be set as
$$E'=(\xymatrix{\mathbb{Z}^\sharp\ar[r]^{\hspace{-0.5cm} j} & \Q^\sharp\bigoplus H^\sharp \ar[r]^{\hspace{-0.5cm} p} & (\Q/\Z)^\sharp \bigoplus H^\sharp}\,).$$
We define the pointed section ${\bf s}\colon (\Q/\Z)^\sharp\bigoplus H^\sharp\to \Q^\sharp\bigoplus H^\sharp$
as $${\bf s}(x+\Z,h)~=~({\bf s}_{(\Q/\Z)}(x+\Z),h).$$
Since each component of ${\bf s}$ is $\sharp$-continuous, it follows that ${\bf s}$ is also $\sharp$-continuous.\\
(ii) For the general case, we have a group $G$ such that $\Z\leq G\lneqq \R$. As in Proposition \ref{Q2}, we define
${\bf s}_{(G/\Z)}\colon G/\Z\to G$ as ${\bf s}_{(G/\Z)}={\bf s}_{(\R/\Z)\, |(G/\Z)}$. {It is proven in a similar way 
to how it was done in Corollary \ref{Q1} that  ${\bf s}_{(G/\Z)}$ is well defined} and, as a consequence, 
$\sharp$-continuous since the $\sharp$-topology is inherited by subgroups.
\end{proof}

\begin{remark} Note that  we have seen in Remark \ref{nows} that $(\mathbb{Z} \to \mathbb{R} \to \R/\Z)$ is not a $ws$-extension.
\end{remark}

As a consequence of Propositions \ref{prducts} and \ref{R1}, we obtain some new examples of exact sequences that $ws$-split.

\begin{corollary} Let $G$ be a subgroup of $\R$ that contains $\Z$, $I$ an infinite index set, and $J\subseteq I$. The following
extensions $ws$-split.
\begin{enumerate}%[(a)]
  \item[(a)]\ \ $E_1=(\xymatrix{\mathbb{Z}^J\ar[r]^{j}& G^I\ar[r]^{\hspace{-1cm} p}& (G/ \mathbb{Z})^J\bigoplus G^{I\setminus J}}).$
  \item[(b)]\ \ $E_2=(\xymatrix{(\mathbb{Z}^n)^\sharp\ar[r]^{j} &(\R^n)^\sharp\ar[r]^{\hspace{-0.4cm} p}\, & (\R^n/ \mathbb{Z}^n)^\sharp}).$
  \item[(c)]\ \ $E_3=(\xymatrix{(\mathbb{Z}^{J})^\sharp \ar[r]^{j} & (G^I)^\sharp\ar[r]^{\hspace{-1cm} p} & (G^J/ \mathbb{Z}^J)^\sharp\bigoplus
  (G^{I\setminus J})^\sharp}),$ where $J$ is finite.
  \item[(d)]\ \ $E=(\xymatrix{(\mathbb{Z}^J)^\sharp \ar[r]^{j} & (\R^I)^\sharp\ar[r]^{\hspace{-1cm} p} & (\R^J/ \mathbb{Z}^J)^\sharp\bigoplus (\R^{I\setminus J})^\sharp}),$ where $J$ is finite.
\end{enumerate}
\end{corollary}
\bigskip

We have seen above that the product of $ws$-extensions define a new $ws$-extension. From this fact, it follows that
the same holds true for direct sums of $ws$-extensions when they are equipped with the product topology.
However, there is another canonical topology that can be defined on a direct sum, namely, the coproduct topology.
We do not know whether the direct sums of $ws$-extensions, equipped with the coproduct topology, define a new $ws$-extension.
Therefore, the following question remains~open:
\medskip

\begin{question}\label{q1} Let  $$\{ E_i=(\xymatrix{B_i \ar[r]^{j_i} & X_i  \ar[r]^{p_i} & A_i})\}_{i\in I}$$ be a family of  $ws$-extensions in {\bf TA}. Is the sequence
$$E=(\xymatrix{\bigoplus\limits_{i\in I} B_i \ar[r]^{j} & \bigoplus\limits_{i\in I}X_i   \ar[r]^{p} & \bigoplus\limits_{i\in I}A_i} ),$$ where $j((b_i)):=(j_i(b_i))$, $p((x_i)):=(p_i(x_i))$ and all direct sums are equipped with the coproduct topology,
a $ws$-extension?
\end{question}

{Although the Bohr topology of a direct sum of Abelian groups and  does not coincide with the coproduct topology in general,
it provides us with the following clarifying example:}
\medskip

\begin{example}
The extension in $\A$:  $$\xymatrix{\mathbb{Z}/ 2\mathbb{Z} \ar[r]^{\alpha} &\mathbb{Z}/ 4 \mathbb{Z} \ar[r]^{\beta} &\mathbb{Z}/ 2\mathbb{Z}},$$
  where $\alpha(1+ 2\mathbb{Z})=2+4 \mathbb{Z} $, $\beta (1+4 \mathbb{Z} )= 1+2 \mathbb{Z}$
   has a pointed section ${\bf{s}}(1+2 \mathbb{Z} )= 1+4 \mathbb{Z}$, and ${\bf{s}}(0+4 \mathbb{Z} )= 0+2 \mathbb{Z}$.
 This implies that the induced extension
 $$\xymatrix{\oplus_{\mathbb{N}}\mathbb{Z}/ 2\mathbb{Z} \ar[r]^{ \oplus_{\mathbb{N}}\alpha} & \oplus_{\mathbb{N}}\mathbb{Z}/  \mathbb{Z} \ar[r]^{ \oplus_{\mathbb{N}}\beta} & \oplus_{\mathbb{N}}\mathbb{Z}/ 2\mathbb{Z}  }$$ also has  a pointed section.
Note that in this case, since $\mathbb{Z}/ 2\mathbb{Z}, \mathbb{Z}/ 4 \mathbb{Z} $ are discrete groups, the coproduct topology agrees with the discrete topology.
Now, if  we take the Bohr topologies, we have the following  extension  in $\TA$:
$$E=(\xymatrix{ (\oplus_{\mathbb{N}}\mathbb{Z}/ 2\mathbb{Z})^\sharp \ar[r]^{ (\oplus_{\mathbb{N}}\alpha)^\sharp} & (\oplus_{\mathbb{N}}\mathbb{Z}/ 4 \mathbb{Z})^\sharp \ar[r]^{( \oplus_{\mathbb{N}}\beta)^\sharp} & (\oplus_{\mathbb{N}}\mathbb{Z}/ 2\mathbb{Z} )^\sharp}),$$
which is not a $ws$-extension (see \cite{COMFORT2001215}).
This fact is relevant because, first, all groups in the extension
$E$ are $0$-dimensional (see Remark \ref{nows}) and, second, it provides an example of two topological  groups $A^\sharp$ and $B^\sharp$ where the natural transformation
$E_{{\A}^\sharp}^{ws}(A^\sharp,B^\sharp) \to E_{\A}(A,B)$ is not surjective, where ${\A}^\sharp$ denotes the full subcategory of $\TA$ whose objects $C^\sharp$  are Abelian groups $C$  provided with the Bohr topology,  and $E_{{\A}^\sharp}^{ws}(A,B)$ is obtained  by taking the  category of   $ws$-extensions of $B^\sharp$ by $A^\sharp$ in ${\A}^\sharp$  modulo topological~isomorphisms.
\end{example}

In light of this example, it would be desirable to know the answer to the following: \medskip

\begin{question}\label{q2} Let  $$\{ E_i=(\xymatrix{B_i \ar[r]^{j_i} & X_i  \ar[r]^{p_i} & A_i})\}_{i\in I},$$ be a family of  extensions in
{\bf A}. When is the sequence
$$E=(\xymatrix{(\bigoplus\limits_{i\in I} B_i)^\sharp \ar[r]^{j} & (\bigoplus\limits_{i\in I}X_i )^\sharp  \ar[r]^{p} &
(\bigoplus\limits_{i\in I}A_i)^\sharp}),$$ where $j((b_i)):=(j_i(b_i))$ and $p((x_i)):=(p_i(x_i))$,
a $ws$-extension?
\end{question}

\section{Remarks, conclusions, and future work}

For two topological  Abelian groups $A, B \in  \TA$ we have analyzed the structure of the Abelian group $E_{\TA}^{ws}(A,B)$
of $ws$-extensions of $B$ by $A$ modulo topological isomorphism of $ws$-extensions, and we have given two different descriptions of this Abelian group:
$$\Sigma_{A,B}/\sim,$$ using all different structures of the topological Abelian group on the product space $B \times A$, and
$$\mathcal{Z}_c(A,B)/\mathcal{B}_c(A,B),$$
using  cocycles  of the form $ f \colon A \times A \to B$ where $f$ is a suitable continuous pointed map.

{We have also provided  a six-term exact sequence to study the canonical map
$E_{\TA}^{ws}(A,B) \to E_{\A}(A,B)$.}

We remark that for two Abelian groups $A, B \in \A$,  $E(A,B)$ can also be obtained as a derived functor using injective or projective resolutions; see \cite{Eilenberg42,Fuchs2015,Hilton-Stammbach}. If $A, B$ are  locally compact topological Abelian groups,
%($A, B \in \mathcal{L}$ )
$E_{\TA}(A,B)$  can also be obtained as a derived functor with the homological algebra studied by~Moskowitz and Hoffman for the subcategory  locally compact topological Abelian groups; see~\cite{BELLO20161607, HOFFMANN2007504, f93a063e-b6b5-3414-a1b5-8af16db46d1a}.

One of the pending objectives is  to obtain $E_{\TA}^{ws}(A,B)$ as a derived functor by taking suitable projective resolutions that can be  constructed using free topological Abelian groups over pointed topological spaces. Note that the category of topological Abelian groups is isomorphic to the category of topological modules over the topological discrete ring $\mathbb{Z}$.
Nummela announced in the paper   \cite{Nummela73} a study of Ext functors for the category of modules over a topological ring using projective resolutions. However, this study appears to have not been published in the end.

We also remark that there are some interesting open problems related to Bohr topologies and $ws$-extensions, as  seen in open Questions  \ref{q1} and \ref{q2}.

%\section*{Use of AI tools declaration}
%The authors declare they have not used Artificial Intelligence (AI) tools in the creation of this article.

\section*{Acknowledgement}

Mar\'{\i}a V. Ferrer and  Salvador Hern\'andez-Mu\~noz acknowledge partial support by the Universitat Jaume I, grant UJI-B2022-39.  Luis J. Hern\'andez acknowledges support from project REGI25/61 of the University of La Rioja. We thank Professor Javier Trigos-Arrieta for several comments and remarks that helped us to improve the manuscript.

%\section*{Conflict of interest}
%
%The authors declare there is no conflict of interest.

\end{document}